\documentclass[12pt]{amsart}
\usepackage{latexsym,fancyhdr,amssymb,color,amsmath,amsthm,graphicx,listings,comment}
\usepackage[section]{placeins}
\newtheorem{thm}{Theorem}

\newtheorem{coro}{Corollary}
\let\paragraph\subsection
\title{Spectral monotonicity of the Hodge Laplacian} 

\author{Oliver Knill}
\date{April 2, 2023}
\address{Department of Mathematics \\ Harvard University \\ Cambridge, MA, 02138 }


\begin{document}
\maketitle

\begin{abstract}
If $K \subset G$ are finite abstract simplicial complexes,
then the eigenvalues of the Hodge Laplacians satisfy 
$\lambda_k(K) \leq \lambda_k(G)$ if padded left. 
\end{abstract}

\section{The theorem}

\paragraph{}
If $d$ is the {\bf exterior derivative} of a {\bf finite abstract simplicial complex} $G$
of $n$ elements, then $L=D^2=(d+d^*)^2=d d^* + d^* d$ is the {\bf Hodge Laplacian} of $G$.
It is a block diagonal $n \times n$ matrix in which the blocks are the
{\bf k-form Laplacians} $L_k = d_k^* d_k + d_{k-1} d_{k-1}^*$ of size $f_k \times f_k$, where
$f=(f_0,f_1,\dots,f_d)$ is the {\bf $f$-vector} of $G$, counting the number $f_k$ of elements in $G$ of 
cardinality $k+1$. Ordering the elements in $x$ and the position in the list $G$ is a choice of coordinates 
and fixes the matrices. The Betti vector $\vec{b} = (b_0,b_1, \dots )$ with {\bf Betti numbers} 
$b_k = {\rm dim}({\rm ker}(L_k))$ does not depend on the ordering. 

\paragraph{}
The symmetric $n \times n$ matrix $L$ is positive semi-definite 
because $L=D^2$ with a symmetric matrix $D$. The {\bf Dirac matrix} $D=d+d^*$
has pairs of positive and negative eigenvalues and the diagonal matrix 
$P(x)=(-1)^{{\rm dim}(x)}$ maps an eigenvector $v$ to the eigenvalue $\lambda$
to $P v$ which is an eigenvector to the eigenvalue $-\lambda$. 
So, also the non-zero eigenvalues of $L$ come in pairs. But more is true: the 
{\bf McKean-Singer symmetry} \cite{McKeanSinger} shows that
the non-zero eigenvalues are distributed equally on even and odd forms. This is 
usually written as an equality for the Euler characteristic ${\rm str}(e^{-t L})={\rm str}(1)$, 
where ${\rm str}(L)= \sum_x (-1)^{{\rm dim}(x)} L_{xx}$. 
We write the eigenvalues $\lambda_k(G)$ of $L$ in ascending order. 

\paragraph{}
If $K$ is a sub-simplicial complex of $G$ with 
$m \leq n$ elements, define $\lambda_{k}(K)=0$ for $k \leq n-m$ and 
$\lambda_{n-m+k}(K) = \mu_k(K)$, where $\mu_k$ are the $m$ eigenvalues of the 
Hodge Laplacian of $K$, again ordered in {\bf ascending order}. The spectra of
$K$ and $G$ can now be compared, when seen as {\bf left-padded} non-descending sequences.
There are $m$ eigenvalues in $K$ and $n$ eigenvalues in $G$ but after padding, the 
sequences both have $n$ elements. Different padding had us miss at first the
following result: 

\begin{thm}[Spectral monotonicity]
\label{1}
$\lambda_j(K) \leq \lambda_j(G)$ for all $j \leq n$.
\end{thm}
\begin{proof}
The proof parallels the case of the Kirchhoff Laplacian $L_0$ \cite{Spielman2009}. 
$G$ is a finite set of non-empty sets closed under the operation of taking finite non-empty subsets.
A set $x \in G$ is called {\bf locally maximal} if it is not contained in an other
simplex. This means that the set $U=\{ x \}$ is an open set in the non-Hausdorff
{\bf Alexandroff topology} $\mathcal{O}$ on $G$ generated by the basis formed by the 
{\bf stars} $U(x) = \{ y \in G, x \subset y \}$. 
If we add a locally maximal simplex to a given complex, the spectrum changes monotonically.
Also the Hodge Laplacian $L(U)= D(U)^2$ is the square of the Dirac operator $D(U)$ of $U$.
If $K \subset G$, then $L(K) \leq L(G)$ in the {\bf Loewner partial order}. One can see this as follows:
if $u: G \to \mathbb{R}$ is a vector, then the Laplacian quadratic form is
is $\langle u, L u \rangle = \langle u, D^2 u \rangle = \langle Du, Du \rangle = ||Du||^2$
and adding new maximal simplices only can increases this quadratic form, also when restricted to any
linear subspace. The {\bf Courant Fischer theorem} then gives using 
$\mathcal{S}_k = \{ V \subset \mathbb{R}^n, {\rm dim}(V)=k \}$
$$ \lambda_k(K) = {\rm min}_{V \in \mathcal{S}_k} {\rm max}_{|u|=1, u \in V} \langle u, L(K) u \rangle
 \leq {\rm min}_{V \in \mathcal{S}_k} {\rm max}_{|u|=1, u \in V} \langle u, L(G) u \rangle 
  = \lambda_k(G) \; . $$
One can see this also in the context of the {\bf interlace theorem} applied to $D$ as the Dirac matrix of $K$
is obtained from the Dirac matrix of $L$ by deleting the row and column belonging to the element $x$ which 
was added. The eigenvalues of the Dirac matrix $D_K$ of $K$ are now interlacing the eigenvalues of
the Dirac matrix $D_G$ of $G$.
\end{proof} 

\paragraph{}
A similar spectral inequality holds also for open subsets of $G$:
$\lambda_j(U) \leq \lambda_j(G)$ for all $j \leq n$. The argument is the same.

\paragraph{}
As we can write down the complex $G$ in a lexicographic order (one an think of $x \in G$ as 
a ``word", then sort the elements like in a dictionary with shorter words coming first), 
where the {\bf vertex set} $V=\bigcup_{x \in G} x$ is totally ordered and ${\rm dim}(x)$ produces 
a partial order, then the last element in the list is always locally maximal. 
The sequence of subsets $G_k = \{ x_1, \dots, x_k \}$ subset $G=\{x_1, \dots, x_n \}$ is now
a {\bf Morse filtration}. The step $G_k \to G_{k+1}$ augments
$f_{{\rm dim}(x_{k+1})}$ by one and either increases $b_{{\rm dim}(x_{k+1})}$ or decreases
$b_{{\rm dim}(x_{k+1})-1}$ assuring that the super sums $\sum_k (-1)^k f_k$ and $\sum_k (-1)^k b_k$
always stay the same. As every new $k$-simplex is glued naturally to a $(k-1)$-dimensional sphere,
this also illustrates that every simplicial complex is naturally an {\bf finite abstract CW complex}. 
(we made use of this structure in \cite{KnillEnergy2020}. Like finite abstract simplicial complexes,
there is never any geometric realization involved).
The CW structure of $G$ contains more information than $G$ as it also uses an order on $G$. 
There are lots of different CW realizations of a given simplicial complex.  
Before each extension $G \to G \cup \{x\}$ we can for example reshuffle the total order on 
$V=\bigcup_{x \in G} x$. One can also see the CW complex in terms of a Morse function $f$
which enumerates the elements of $G$ such that $\{ f \leq k \}=G_k$. 

\paragraph{}
A closed sub complex $K \subset G$ of $G$ can be seen as a {\bf closed set} in the finite topology 
$\mathcal{O}$ generated by the stars. The complement $U$ is open. 
The exterior derivative can also be defined for the open set $U=G \setminus K$. 
Having a {\bf cohomology for open sets} is a
new feature in the discrete. Unlike the case when simplicial complexes are realized in the continuum 
the topology $\mathcal{O}$ is non-Hausdorff. It is Zariski type because closed subsets are 
sub-simplicial complexes. The ability to split spaces $G$ into an open set $U$ and a closed set $K$
gives a new approach to Meyer-Vietoris as we have no overlap. We will write about
this elewhere. 

\paragraph{}
For any open or closed set and Hodge Laplacian $L$, we have a 
{\bf Betti vector} $\vec{b}=(b_0,b_1,\dots,b_d)$ where $b_k$ is the nullity of the $k$'th block
in $L=L_0 \oplus L_1 \oplus \dots L_d$. We also have the {\bf $f$-vector} $\vec{f} = (f_0,f_1,\dots,f_d)$
where the integers $f_k$ count the number of $k$-dimensional parts in the complex. 
The {\bf Euler-Poincar\'e formula}
$\chi = \sum_k (-1)^k f_k = \sum_k (-1)^k b_k$ holds both for open and closed sets. 
Note that in the case of open sets, we sometimes have to deal with 
{\bf Hodge blocks} that are $0 \times 0$ matrices which have empty spectrum. 

\paragraph{}
We can easily see for example that every {\bf $f$-vector} $\vec{f} \geq 0$ 
and every {\bf Betti vector} $\vec{b} \geq 0$ can be realized by open sets. For closed sets, 
meaning sub-simplicial complexes, we do not know to characterized the set of possible Betti vectors 
but for the $f$-vectors, there is the {\bf Kruskal-Katona} characterization. 
In order to realize a given $f$-vector $\vec{f} \geq 0$ with an open set, just take a disjoint union $U$ of 
non-empty sets with $f_k$ sets of cardinality $k+1$, then $\vec{b}(U)=\vec{f}(U)$ agrees with the 
pre-described $\vec{f}$ and in this case even $\vec{b}=\vec{f}$. The open set
$U=\{ \{0\},\{1\},\{2\}, \{3,4 \}, \{5,6,7,8\}, \{9,10,11,12 \} \}$ for example realizes 
$\vec{f} = (3,1,0,2)$. The Hodge Laplacian is a $6 \times 6$ matrix with $4$ blocks.
$L_0$ is the $3 \times 3$ zero matrix, $L_1$ is a $1 \times 1$ zero matrix, $L_2$ is the $0 \times 0$ matrix 
and $L_3$ is the $2 \times 2$ zero matrix. 

\paragraph{}
One can now compare the spectrum of $L_{U,K}=L_U \oplus L_K$ in which $U$ and $K$ are 
disjoint and the spectrum of $L_G$,
where the two sets are united. Both matrices $L_{U,K}=L_U \oplus L_K$ and $L_G$ are symmetric
$n \times n$ matrices.  While $\sigma(L_{U,K}) \leq \sigma(L_G)$ is not true in general, we will
see elsewhere that after a {\bf Witten deformation} $d_s=e^{-gs} d e^{gs}$ with a suitable function 
$g:G \to \{0,1\}$ of the exterior derivatives $d_G$ and $d_{U,K} = d_U \oplus d_K$ we can enforce 
this spectral inequality. The Witten deformation does not change the dimensions of the kernels of $L_k$
and so preserves the $0$ spectrum. It can deform however the other eigenvalues. 
Experimentally, we see $\sigma(L_{U,K}) \leq 2\sigma(L_G)$ but this has not been proven.  

\paragraph{}
By choosing the function $g$ supported on the interface set of $K$ and $U$ we can change the focus. 
For large enough $s$, we get then $\sigma(L_{U,K}^{(s)}) \leq \sigma(L_G^{(s)})$. Since the $0$ eigenvalues
do not change under Witten deformation, this especially will establish the {\bf fusion inequality} 
$\vec{b}_U + \vec{b}_K  \geq \vec{b}_G$, reflecting that during a {\bf fusion} of an open and closed set,
new harmonic forms could be generated but that no harmonic forms are lost. In this note we 
leave it at announcing this inequality and hope to discuss it more in a future article. 

\paragraph{}
The monotonicity Theorem~\ref{1} shows that in the case when $U=\{ x\}$ is a open set 
with a single point $x$, where the Hodge Laplacian of $U$ is a $1 \times 1$ matrix  and 
$\sigma(L_U)=\{0\}$ we have $\sigma(L_{U,K}) \leq \sigma(L_G)$. 
For open sets like singleton sets $U=\{x\}$, the k-form
Laplacians with ${\rm dim}(x) \neq k$ are all $1 \times 1$ zero matrices. 

\paragraph{}
A small example, where the fusion equality is an equality
$\vec{b}_U + \vec{b}_K  = \vec{b}_G$ is where a closed interval $K=\{ \{1\},\{2\},\{1,2\} \}$, and 
an open interval $U=\{ \{3\},\{4\},\{2,3\},\{3,4\},\{4,1\} \}$ are merged to a circle 
$G=\{ \{1\},\{2\},\{3\},\{4\}$, $\{1,2\},\{2,3\},\{3,4\},\{4,1\} \}$, which gives $\vec{b}(K)=(1,0)$, 
$\vec{b}(U)=(0,1)$ and $\vec{b}(G)=(1,1)$. This generalizes to arbitrary dimensions: 
a {\bf $d$-sphere} $G$ is the union of an open $d$-ball $U$ with 
$\vec{b}(U) = (0, \cdots,0,1)$ and a closed $d$-ball $K$ with $\vec{b}(K)=(1,0, \dots, 0)$ 
adding to $\vec{b}(G)=(1,0,\dots,0,1)$. 

\paragraph{}
A small example, where the fusion inequality is a strict inequality
$\vec{b}_U + \vec{b}_K > \vec{b}_G$ is with $K=\{ \{1\},\{2\} \}$ and $U=\{ \{1,2\} \}$ which is 
merged to $G=\{ \{1\},\{2\},\{1,2\} \}$ and where $\vec{b}(K)=(2,0)$ and $\vec{b}(U)=(0,1)$ and
$\vec{b}(G) = (1,0)$. We have here $\sigma(L_{K,U}) = \{ 0, 0, 0 \}$
and $\sigma(G) = \{ 0,2,2 \}$. A harmonic 1-form and a harmonic 0-form have merged. 

\paragraph{}
The inequality in the theorem holds on every sector of $k$-forms. 
The spectrum of the $k$-form Laplacian only can change if we add a $(k-1),k$ or $(k+1)$-dimensional 
simplex. In any case, it is good to state the monotonicity restricted to each of the forms 
Laplacians:

\begin{coro}[Form monotonicity]
$\lambda_j(L_{k,K}) \leq \lambda_j(L_{k,G})$ for all $j \leq n$ and $0 \leq k \leq {\rm dim}(G)$. 
\end{coro}

\paragraph{}
Seen as such, the result generalizes the result for the Kirchhoff Laplacian $L_0$, 
usually formulated within graph theory \cite{Spielman2009} and which is based on
spectral monotonicity \cite{HornJohnson2012}.
If $G=(V,F)$ is a subgraph of $K=(V,E)$, then the eigenvalues of the
Kirchhoff Laplacian satisfy $\lambda_k(K) \leq \lambda_k(G)$.
For $1$-dimensional simplicial complex, the $1$-form Laplacian
$L_1=d_0 d_0^*$ is essentially isospectral to $L_0 = d_0^* d_0$. In general,
the {\bf McKean-Singer symmetry} assures that $\bigcup_{k {\rm even}} L_k$ and 
$\bigcup_{k {\rm odd}} L_k$ are essentially isospectral meaning that they have 
the same non-zero eigenvalues. 

\paragraph{}
We made use of the McKean-Singer symmetry to show that the Hodge spectrum
can not determine the simplicial complex \cite{knillmckeansinger} in general. This is not
surprising given that in the continuum, the first counter examples of Milnor to the Kac inverse
spectral question for manifolds gave already Hodge isospectral examples. 
There is a lot to explore still for the inverse spectral problem. Can one read off the 
Betti vectors from the spectrum alone, if one does not know to which k-form sector the 
eigenvalues belong? Can one reconstruct the simplicial complex if one knows the eigenvalues
of all Barycentric refinements of the complex? Can one reconstruct the complex from the 
Hodge spectrum and the connection Laplacian spectrum?

\paragraph{}
The theorem is maybe a bit more surprising when looking at other Laplacians defined for 
a simplicial complex $G$.  The {\bf connection Laplacian} is defined as 
$H_{xy} = 1$ if $x \cap y$ is non-empty and $H_{xy}=0$ else.
In this case, we see no relations between eigenvalues of $K$ and $G$.
While most eigenvalues satisfy $\lambda_j(K) \leq \lambda_j(G)$ this is
not universally true. Also, while $H(G),H(K)$ are unimodular \cite{KnillEnergy2020}
this is not true for open sets: $H(U)$ is in general singular. In general, 
the spectrum of the Hodge Laplacian tends to be larger than the spectrum of the 
connection Laplacian but this is also not universally true. There is a relation between
the connection Laplacian and the Hodge Laplacian for one dimensional complexes
\cite{Hydrogen}: we called this the hydrogen identity $L= H -H^{-1}$. 

\paragraph{}
The {\bf complexity} ${\rm Det}(L(G))$ of a simplicial complex is defined as 
the pseudo determinant \cite{cauchybinet} $\prod_{k, \lambda_k \neq 0} \lambda_k(G)$ 
of the Hodge Laplacian. The Forest quantity ${\rm det}(L(G)+1)$ is also monotone
\cite{cauchybinet,TreeForest}. Also of interest can be the {\bf total energy} 
${\rm tr}(L)=\sum_j \lambda_j$,  
the sum of the eigenvalues of $L$. An immediately consequence of the Theorem~(\ref{1}) is:

\begin{coro}[Complexity is monotone]
If $K \subset G$, then ${\rm Det}(L(K)) \leq {\rm Det}(L(G))$, 
${\rm det}(L(K)+1) \leq {\rm det}(L(G)+1)$ and ${\rm tr}(L(K)) \leq {\rm tr}(L(G))$. 
\end{coro}

\paragraph{}
We even see in all experiments so far that ${\rm Det}(L(G)) \geq {\rm Det}(L(K)) {\rm Det}(L(U)))$ 
and ${\rm tr}(L(G)) \geq {\rm tr}(L(K))  + {\rm tr}(L(U)))$ but this does not 
follow. It is intuitive if we think about the pseudo determinant of $L$
as a measure of {\bf complexity} and the trace as a {\bf total energy} of the complex. 
But these stronger inequalities are not proven. 

\paragraph{}
Here is some code which allow to compute the spectra of open or closed sets
$K$ and its complement $U$. We also add code to compute some basic data like the $f$-vector, 
the Betti vectors, the pseudo determinant, Euler characteristic, the trace or the analytic torsion. 
As the pseudo determinant ${\rm Det}(L_0)$ is the number of rooted spanning trees in 
the graph obtained as the {\bf skeleton complex} $(V,E)$ with $V = \bigcup_{|x|=1} x$
and $E= \bigcup_{|x|=2} x$, one can interpret ${\rm Det}(L)$ as a super count
of higher dimensional trees. This comes up in the context of {\bf analytic torsion} 
$A(G) = \prod_k {\rm Det}(L_k)^{k (-1)^{k+1}}$ \cite{KnillTorsion} which 
agrees with the {\bf super pseudo determinant }
$\prod_k {\rm Det}(D_k)^{(-1)^k}$ of the {\bf Dirac blocks} $D_k=d_k^* d_k$. 

\paragraph{}
For the Kirchhoff matrix $L_0$ we have proven in \cite{Eigenvaluebounds} that
$\lambda_k \leq a_k + a_{k-1}$, where
$a_k$ are the diagonal entries of $L_0$, the vertex degrees and $a_{0}=0$. 
(We use letters $a_k$ here because $d_k, b_k$ have been used already for the 
exterior derivatives or Betti numbers).
This upper bound is not always true in the Hodge case. 
In order to get bounds, one has first to get something like
the Anderson-Morley bound \cite{AndersonMorley1985} for the spectral radius. 
We currently think that $\lambda_k \leq 2a_k + a_{k-1}+a_{k-2}$ will do 
because in the $0$-dimensional case we had only to deal with $L_0=d_0^* d_0$ while
for general $k$-forms we have $L_k = d_k^* d_k + d_{k-1} d_{k-1}^*$. 

\begin{tiny}
\lstset{language=Mathematica} \lstset{frameround=fttt}
\begin{lstlisting}[frame=single]
F[G_]:=Module[{l=Map[Length,G]},If[G=={},{},
 Table[Sum[If[l[[j]]==k,1,0],{j,Length[l]}],{k,Max[l]}]]]; s[x_]:=Signature[x];L=Length;
s[x_,y_]:=If[SubsetQ[x,y]&&(L[x]==L[y]+1),s[Prepend[y,Complement[x,y][[1]]]]*s[x],0];
Dirac[G_]:=Module[{f=F[G],b,d,n=Length[G]},b=Prepend[Table[Sum[f[[l]],{l,k}],{k,Length[f]}],0];
 d=Table[s[G[[i]],G[[j]]],{i,n},{j,n}]; {d+Transpose[d],b}];
Hodge[G_]:=Module[{Q,b,H},{Q,b}=Dirac[G];H=Q.Q;Table[Table[H[[b[[k]]+i,b[[k]]+j]],
 {i,b[[k+1]]-b[[k]]},{j,b[[k+1]]-b[[k]]}],{k,Length[b]-1}]];
Beltrami[G_] := Module[{B=Dirac[G][[1]]},B.B]; nu[A_]:=If[A=={},0,Length[NullSpace[A]]];  
Betti[G_]:=Map[nu,Hodge[G]];
Closure[A_]:=If[A=={},{},Delete[Union[Sort[Flatten[Map[Subsets,A],1]]],1]];Cl=Closure;
Whitney[s_]:=If[Length[EdgeList[s]]==0,Map[{#}&,VertexList[s]],
   Map[Sort,Sort[Cl[FindClique[s,Infinity,All]]]]];
OpenStar[G_,x_]:=Module[{U={}},Do[If[SubsetQ[G[[k]],x],U=Append[U,G[[k]]]],{k,Length[G]}];U];
Basis[G_]:=Table[OpenStar[G,G[[k]]],{k,Length[G]}]; Stars=Basis; 
RandomOpenSet[G_,k_]:=Module[{A=RandomChoice[Basis[G],k],U={}},Do[U=Union[U,A[[j]]],{j,k}];U];
Fvector[G_]:=If[Length[G]==0,{},Delete[BinCounts[Map[Length,G]],1]];
FirstNonZero[f_]:=-(-1)^ArrayRules[Chop[f]][[1,1,1]] ArrayRules[Chop[f]][[1,2]];
PseudoDet[A_]:=If[A=={},1,FirstNonZero[CoefficientList[CharacteristicPolynomial[A,x],x]]];
Fvector[G_,U_,K_]:={"f_G=",Fvector[G],"f_U=",Fvector[U],"f_K=",Fvector[K]};
w[x_]:=-(-1)^Length[x];  dim[x_]:=Length[x]-1; EulerChi[A_]:=Total[Map[w,A]];
EulerChi[G_,U_,K_]:={"X(G)= ",EulerChi[G],"X(K)= ",EulerChi[K],"X(U)= ",EulerChi[U]}; 
Betti[G_,U_,K_]:={"b_G=",Betti[G],"b_U=",Betti[U],"b_K=",Betti[K]};
Torsion[G_]:=Module[{u=Map[PseudoDet,Hodge[G]],k},
  Product[If[EvenQ[k],u[[k]]^(k-1),1/u[[k]]^(k-1)],{k,Length[u]}]];
ConnectionLaplacian[G_]:=Module[{n=Length[G],A=Table[1,{n},{n}]},
  Do[If[DisjointQ[G[[k]],G[[l]]],A[[k,l]]=0],{k,n},{l,n}];A];

G=Whitney[RandomGraph[{20,50}]]; U=RandomOpenSet[G,10];K=Complement[G,U];
KK=Beltrami[K];GG=Beltrami[G]; UU=Beltrami[U]; ev[L_]:=Sort[Eigenvalues[1.0*L]]; 
Print[Betti[G,U,K]]; Print[Fvector[G,U,K]];  Print[EulerChi[G,U,K]];
Print[{"tr(G)= ",Tr[GG]," tr(K)= ",Tr[KK]," tr(U)= ",Tr[UU]}];
Print[{"Det(G)= ",PseudoDet[GG]," Det(K)= ",PseudoDet[KK]," Det(U)= ",PseudoDet[UU]}];
Print[N[PseudoDet[GG]/(PseudoDet[KK]*PseudoDet[UU])]]
Print[{"T(G)= ",Torsion[G]," T(K)= ",Torsion[K]," T(U)= ",Torsion[U]}];

Print["The difference of the spectra is non-negative by Theorem 1   "]; 
{EG,EK,EU}=PadLeft[Map[ev,{GG,KK,UU}]]; ListPlot[{EG-EK,EG-EU},Joined->True]

Print["This holds for each k form by the Corollary    "]; 
m=Length[Hodge[K]];Table[KK = Hodge[K][[k]]; GG = Hodge[G][[k]]; UU = Hodge[U][[k]];
{EG,EK,EU}=PadLeft[Map[ev,{GG,KK,UU}]]; ListPlot[{EG-EK,EG-EU},Joined->True],{k,m}]

Print["The energy inequality for L_{K,U}=L_K + L_U and L_G does not always hold   "]; 
ListPlot[{ev[GG],Sort[Join[ev[KK],ev[UU]]]},Joined->True]

\end{lstlisting}
\end{tiny}

\begin{figure}[!htpb]
\scalebox{1.0}{\includegraphics{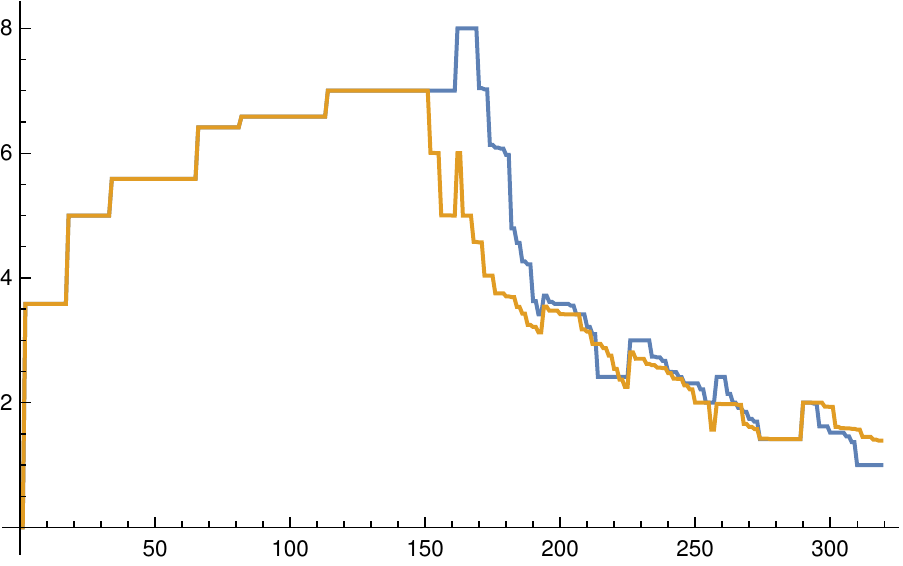}}
\caption{
The difference of the spectra according to the theorem. We see
both the difference $\sigma(L_G)-\sigma(L_K)$ (orange) and 
$\sigma(L_G)-\sigma(L_U)$ (blue) 
}
\end{figure}

\begin{figure}[!htpb]
\scalebox{1.0}{\includegraphics{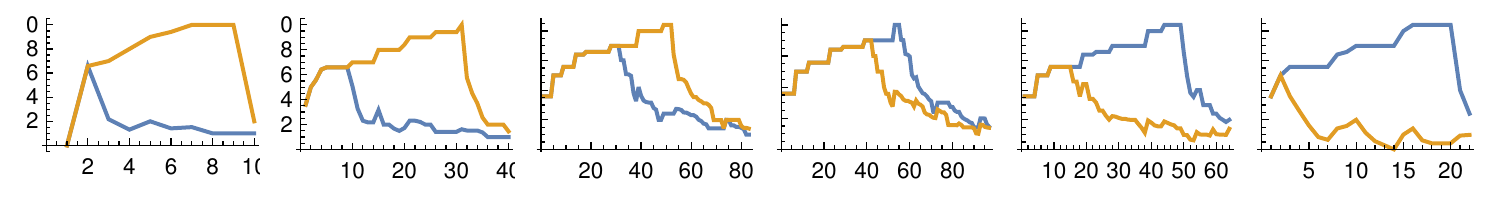}}
\caption{
The difference of the form Laplacians is shown as an 
illustration of the corollary.
In some sectors, we see a dominance of $\sigma(L_G)-\sigma(L_U)$
in others we see that $\sigma(L_G)-\sigma(L_K)$ dominates.
}
\end{figure}

\bibliographystyle{plain}

\end{document}